\documentclass[12pt]{amsart}
\usepackage{amsmath}
\usepackage{amsthm}
\usepackage{amssymb}
\usepackage{amscd}          			%diagrams
\usepackage{bbm}            			%bold numbers
\usepackage{mathrsfs}           		%nice calligraphy
\usepackage{mathalfa}

\usepackage{graphicx}          		%uncomment for inclusion of graphic images

\usepackage{verbatim}           		%DO NOT comment this line
\usepackage{enumitem}      		% ENUMITEM	

\theoremstyle{plain}
\newtheorem{Theorem}{Theorem}

\newtheorem{theorem}[Theorem]{Theorem}

\newtheorem{corollary}[Theorem]{Corollary}

\newtheorem{lemma}[Theorem]{Lemma}

\theoremstyle{definition}
\newtheorem{Definition}[Theorem]{Definition}

\newtheorem{example}[Theorem]{Example}
\newtheorem{definition}[Theorem]{Definition}
\newtheorem{remark}[Theorem]{Remark}

\theoremstyle{remark}

\newcommand{\N}{\mathbb{N}}     %natural numbers
\newcommand{\R}{\mathbb{R}}     %real numbers
 
\def\r{\R}

\def\muplu{\mu^{+}}       %postitive variation of t.m mu
\def\mumin{\mu^{-}}         %negative variation of t.m mu
\def\nuplu{\nu^{+}}       %postitive variation of deficient  t.m nu
\def\numin{\nu^{-}}         %negative variation of deficient  t.m nu
\def\laplu{\la^{+}}  
\def\lamin{\la^{-}}  
\newcommand{\calA}{\mathscr{A}}

\newcommand{\calC}{\mathscr{C}}

\newcommand{\calK}{\mathscr{K}}

\newcommand{\calO}{\mathscr{O}}

\DeclareMathOperator{\cl}{cl}				% closure
\newcommand{\TM}{TM} % the space of all topological measures on X
\newcommand{\M}{M} % the space of all measures on X

\newcommand{\STM}{\mathbf{STM}} % the space of all signed topological measures on X
\newcommand{\SDTM}{\mathbf{SDTM}} % the space of all signed  deficient topological measures on X
\newcommand{\DTM}{DTM} % the space of all deficient topological measures on X

\newcommand{\ox}{\calO(X)}
\newcommand{\cx}{\calC(X)}
\newcommand{\kx}{\calK(X)}

\newcommand{\ax}{\calA(X)}

\newcommand{\bccx}{\calK_{c}(X)}

\newcommand{\bcx}{\kx}

\newcommand{\bcsx}{\calK_{s}(X)}
\newcommand{\ksx}{\calK_{s}(X)}
\newcommand{\bosx}{\calO_{s}^{*}(X)}
\newcommand{\bcox}{\calK_{0}(X)}

\def\E{{{\mathcal{E}}}}

\def\C{{{\mathcal{C}}}}

\renewcommand{\O}{\emptyset}

\def\E{\mathcal{E}}

\def\sm{\setminus}
\def\cl{\overline}

\def\se{\subseteq}
\def\sc{\sqcup}
\def\bsc{\bigsqcup}

\def\eps{\epsilon}
\def\norm{\parallel}
\def\la1{\lambda_1}
\def\la2{\lambda_2}
\def\la0{\lambda_{0}}
\def\la{\lambda}

\usepackage{times}

\begin{document}
%\tableofcontents
\title{Signed topological measures on locally compact spaces}
\author{S. V. Butler, University of California, Santa Barbara } 
\address{Department of Mathematics,  University of California, Santa Barbara,  552 University Rd, Isla Vista, CA 93117, USA}
\email{svtbutler@ucsb.edu}
\date{February 19, 2019}
\maketitle
\begin{abstract}
In this paper we define and study signed deficient topological measures and signed topological measures (which generalize signed measures) 
on locally compact spaces.
We prove that a signed deficient topological measure is $\tau$-smooth on open sets and 
$\tau$-smooth on compact sets. We show that the family of signed measures that are differences of two Radon measures 
is properly contained in the family of signed topological measures, which in turn is properly contained in the family of signed deficient topological measures. 
Extending known results for compact spaces, 
we prove that a signed topological measure is the difference of its positive and negative variations if at least one variation is finite; 
we also show that the total variation is the sum of the positive and negative variations.
If the space is locally compact, connected, locally connected, and has the Alexandroff one-point compactification of genus 0,    
a signed topological measure of finite norm can be represented as a difference of two topological measures. 
\end{abstract}

\section{Introduction}

The study of topological measures (initially called quasi-measures) began with papers by 
J. F. Aarnes \cite{Aarnes:TheFirstPaper}, \cite{Aarnes:Pure}, and \cite{Aarnes:ConstructionPaper}.
There are now many papers devoted to topological measures and corresponding non-linear functionals; 
their application to symplectic topology has been studied in numerous papers (beginning with \cite{EntovPolterovich}) 
and a monograph (\cite{PoltRosenBook}). 
The natural generalizations of topological measures are signed topological measures and deficient topological measures.
Signed topological measures of finite norm on a compact space were introduced in  \cite{Grubb:Signed} 
then studied and used in various works, including \cite{Grubb:SignedqmDimTheory}, 
\cite{OrjanAlf:CostrPropQlf},  \cite{Svistula:Signed},  and \cite{Svistula:Convergence}. 
Deficient topological measures (as real-valued functions on a compact space) were first defined and used by 
A. Rustad and O. Johansen in \cite{OrjanAlf:CostrPropQlf} and later independently rediscovered and further developed by  M. Svistula in \cite{Svistula:Signed} 
and  \cite{Svistula:DTM}.    
In this paper we define and study signed deficient topological measures and signed topological measures on locally compact spaces. 
These set functions may assume $\infty$ or $-\infty$. We prove that a signed deficient topological measure is $\tau$-smooth on open sets and 
$\tau$-smooth on compact sets. We show that the family of signed measures that are differences of two Radon measures 
is properly contained in the family of signed topological measures, which in turn is properly contained in the family of signed deficient topological measures. 
Extending known results for compact spaces, 
we prove that a signed topological measure is the difference of its positive and negative variations if at least one variation is finite.
We also show that the total variation is the sum of the positive and negative variations.
If the space is locally compact, connected, locally connected, and has the Alexandroff one-point compactification of genus 0,  
we prove that a signed topological measure of finite norm can be represented  as the difference of two topological measures. This representation is not unique.
 
In this paper $X$ will be a locally compact space. By 
$\ox$ we denote the collection of open subsets of   $X $;
by $\cx$  the collection of closed subsets of   $X $;
by $\kx$  the collection of compact subsets of   $X $. We also set
$ \ax = \cx \cup \ox$. 
%By $\bcox$ we denote the collection of finite unions of disjoint compact connected sets.
We denote by $\cl E$ the closure of a set $E$, and  $ \bsc$ stands for a union of disjoint sets.
We say that a signed set function is real-valued if its range is $ \R$. 
When we consider set functions into extended real numbers they are not identically $ \infty$ or $- \infty$. 

\begin{definition} \label{MDe2}
Let $X$ be a  topological space and $\mu$ be a set function on $\E$, a family of subsets of $X$ that contains $\ox \cup \cx$. 
We say that 
\begin{itemize}
\item
$\mu$ is compact-finite if $ |\mu(K) | < \infty$ for any $ K \in \kx$.
\item
$\mu$ is $\tau-$ smooth on compact sets if 
$K_\alpha \searrow K, K_\alpha, K \in \kx$ implies $\mu(K_\alpha) \rightarrow \mu(K)$.
\item
$\mu$ is $\tau-$ smooth on open sets if 
$U_\alpha \nearrow U, U_\alpha, U \in \ox$ implies $\mu(U_\alpha) \rightarrow \mu(U)$.
\item
$\mu$ is simple if it only assumes  values $0$ and $1$.
\end{itemize}
\end{definition}

\noindent
We recall the following easy lemma which can be found, for example, in 
\cite{Halmos} (see Chapter X, par. 50, Theorem A).
\begin{lemma} \label{HalmEzLe}
If $ C \se U \cup V$, where $C$ is compact, $U, V$ are open, then there exist compact sets 
$K$ and $D$ such that $C = K \cup  D, \ K \se U, \ \ D \se V$.
\end{lemma}

\noindent
Recall the following fact. (see, for example, \cite{Dugundji}, Chapter XI, 6.2):
\begin{lemma} \label{easyLeLC}
Let $K \subseteq U, \ K \in \bcx,  \ U \in \ox$ in a locally compact space $X$.
Then there exists a set  $V \in \ox$ with compact closure such that
$ K \se V \se \cl V \se U. $ 
\end{lemma}

\begin{remark} \label{netsSETS}
Here is an observation which follows, for example, from Corollary 3.1.5 in \cite{Engelking}. 
\begin{itemize}
\item[(i)]
If $K_\alpha \searrow K, K \se U,$ where $U \in \ox,\  K, K_\alpha \in \cx$, and $K$ and at least one of $K_\alpha$ are compact, 
then there exists $\alpha_0$ such that
$ K_\alpha \se U$ for all $\alpha \ge \alpha_0$.
\item[(ii)]
If $U_\alpha \nearrow U, K \se U,$ where $K \in \kx, \ U, \ U_\alpha \in \ox$ then there exists $\alpha_0$ such that
$ K \se U_\alpha$ for all $\alpha \ge \alpha_0$.
\end{itemize}
\end{remark}

\begin{Definition}\label{DTM}
A  deficient topological measure on a locally compact space $X$ is a set function
$\nu:  \cx \cup \ox \longrightarrow [0, \infty]$ 
which is finitely additive on compact sets, inner compact regular, and 
outer regular, i.e. :
\begin{enumerate}[label=(DTM\arabic*),ref=(DTM\arabic*)]
\item \label{DTM1}
if $C \cap K = \O, \ C,K \in \kx$ then $\nu(C \sc K) = \nu(C) + \nu(K)$; 
\item \label {DTM2} 
$ \nu(U) = \sup\{ \nu(C) : \ C \se U, \ C \in \kx \} $
 for $U\in\ox$;
\item \label{DTM3} 
$ \nu(F) = \inf\{ \nu(U) : \ F \se U, \ U \in \ox \} $  for  $F \in \cx$.
\end{enumerate}
We denote by $\DTM(X)$ the collection of all deficient topological measures on $X$.
We say that a deficient topological measure $ \nu$ is finite if $ \nu(X) < \infty$.
\end{Definition} 

\noindent
For a closed set $F$, $ \nu(F) = \infty$ iff $ \nu(U) = \infty$ for every open set $U$ containing $F$.

\begin{Definition}\label{TMLC}
A topological measure on $X$ is a set function
$\mu:  \cx \cup \ox \to [0,\infty]$ satisfying the following conditions:
\begin{enumerate}[label=(TM\arabic*),ref=(TM\arabic*)]
\item \label{TM1} 
if $A,B, A \sc B \in \kx \cup \ox $ then
$
\mu(A\sqcup B)=\mu(A)+\mu(B);
$
\item \label{TM2}  
$
\mu(U)=\sup\{\mu(K):K \in \bcx, \  K \se U\}
$ for $U\in\ox$;
\item \label{TM3}
$
\mu(F)=\inf\{\mu(U):U \in \ox, \ F \se U\}
$ for  $F \in \cx$.
\end{enumerate}
We denote by $\TM(X)$ the collection of all topological measures on $X$.
\end{Definition} 

\noindent
The following Definition is from \cite{Butler:DTMLC}, Section 2.

\begin{definition} \label{laplu}
Given signed set function $\la: \kx  \longrightarrow [-\infty, \infty] $ which assumes at most one of $ \infty, -\infty$
we define two set functions on $\ox \cup \cx$, 
the positive variation $\laplu$ and the total variation $| \la|$,  
as follows:  \\
for an open subset $U \se X$ let 
\begin{eqnarray} 
%\label{modla open set}
\laplu(U) = \sup \{\la(K): \  K \se U,  K \in \kx \}; 
\end{eqnarray}
\begin{eqnarray} 
%\label{modnu open set}
|\la| (U) = \sup \{ \sum_{i=1}^n |\la(K_i)| : \  \bsc_{i=1}^n K_i \se U, \ K_i \se \kx,  \, n \in \N \}; 
\end{eqnarray}
and for a closed subset $F \se X$ let 
\begin{eqnarray}
% \label{modla closed set}
\laplu(F)  = \inf\{ \laplu (U) : \ F \se U, \ U \in \ox\}. 
\end{eqnarray}
\begin{eqnarray} 
%\label{modnu closed set}
|\la| (F)  = \inf\{ |\la| (U) : \ F \se U, \ U \in \ox\}. 
\end{eqnarray} 
We define the negative variation $\lamin$ associated with a signed set function $\la$  
as a  set function $\lamin = (- \la)^{+}$.  
\end{definition}

One may consult  \cite{Butler:DTMLC} for more properties of deficient topological measures on locally compact spaces, 
including monotonicity and superadditivity, 
as well as more information about $ \laplu, \lamin$ and $ | \la|$. 

\section{Signed deficient topological measures} \label{seSDTM}

\begin{Definition}\label{SDTM}
A signed deficient topological measure on a locally compact space $X$ is a set function
$\nu:  \cx \cup \ox \rightarrow [ - \infty, \infty ] $  that assumes at most one of $\infty, 
-\infty$ and that is finitely additive on compact sets, inner compact regular on open sets, and outer regular on closed sets, 
i.e. 
\begin{enumerate}[label=(SDTM\arabic*),ref=(SDTM\arabic*)]
\item \label{SDTM1}
If $C \cap K = \O, \ C,K \in \kx$ then $\nu(C \sc K) = \nu(C) + \nu(K);$ 
\item \label{SDTM2} 
$\mu(U)=\lim\{\mu(K):K \in \bcx, \  K \se U\}
$ for $U\in\ox$;
\item \label{SDTM3}
$\mu(F)=\lim\{\mu(U):U \in \ox, \ F \se U\}$ for  $F \in \cx$.
\end{enumerate}
A signed deficient topological measure is compact-finite if $| \nu(K) | <\infty$ for each $K \in \kx$.
By $SDTM(X)$ we denote the collection of all signed deficient topological measures on $X$. 
\end{Definition} 

\begin{remark} \label{netcond}

In condition \ref{SDTM2} we mean the limit of the net $\nu(C)$ with the index set $\{ C \in \kx: \ C \se U\}$ ordered
by inclusion. The limit exists and is equal to $\nu(U)$. Condition  \ref{SDTM3} is interpreted in a similar way, 
with the index set  being $\{ U \in \ox: \ U \supseteq C \}$ ordered by reverse inclusion.
\end{remark} 

\begin{remark} \label{byCompacts}
Since we consider set-functions that are not identically $ \infty $ or $ - \infty$, we see that for a signed deficient topological measure $ \nu (\O) = 0$.
If $\nu$ and $\mu$ are signed deficient topological measures that agree on $\kx$, then $\nu =\mu$;
if $\nu \le \mu$ on $\kx$, then $\nu  \le \mu$.
\end{remark}

\begin{remark} \label{DTMsdtm}
Any deficient topological measure is a signed deficient topological measure.
%Any (compact-finite) topological measure is a (compact-finite) signed deficient topological measure. 
\end{remark}

\begin{lemma} \label{finAddOC}
Let $\mu : \kx \cup \ox \longrightarrow [- \infty, \infty]$  be a set function 
that assumes at most one of $\infty, 
-\infty$ and such that 
\begin{enumerate}[label=(a\arabic*),ref=(a\arabic*)]
\item
$\mu(U)=\lim\{\mu(K):K \in \bcx, \  K \se U\}
$ for $U\in\ox$;
\item \label{SDTM3a}
$\mu(K)=\lim\{\mu(U):U \in \ox, \ K \se U\}$ for  $K \in \kx$.
\end{enumerate}
Then $\mu$ is finitely additive on compact sets iff it is finitely additive on open sets. 
In particular, this holds for a signed deficient topological measure. 
\end{lemma}

\begin{proof}
Without loss of generality, assume that $\mu$ does not assume $- \infty$. 
Suppose $ \mu$ is finitely additive on compact sets. 
Let $U_1 \sc U_2 = U$, where $U_1, U_2, U \in \ox$.

First, we shall show that if at least one of $\mu(U_1), \mu(U_2)$ is $ \infty$, then also $ \mu(U) = \infty$; in this case 
the finite additivity on open sets trivially holds.
So let $ \mu(U_1) = \infty$. Suppose to the contrary that $ \mu(U)  < \infty $. For $ \eps = 1$ let $ C \in \kx$ be such that
$ C \se U$ and $| \mu(U) - \mu(K) | < 1$ for any compact set $K$  satisfying $ C \se K \se U$.
Choose  $n \in \N$ such that $ n > | \mu(U) |$ if $ \mu(U_2) = \infty$, 
or $n$ such that $ n + \mu(U_2) - 1 > \mu(U) + 1$ if $ \mu(U_2) \in \r$.
Pick a compact $C_1 \se U_1$ such that $ \mu(C_1) > n$. Pick a compact $ C_2 \se U_2$ such that  
$ \mu(C_2) > n$ if $ \mu(U_2) = \infty$, and
$ | \mu(C_2)  - \mu(U_2) | < 1$ if $\mu(U_2) < \infty$.
We may assume by Lemma \ref{HalmEzLe} that  $ C \se  C_1 \sc C_2 \se U$, 
so $\mu(C_1 \sc C_2) \le \mu(U) + 1$. But also 
$\mu(C_1 \sc C_2) = \mu(C_1) + \mu(C_2)  > \mu(U) + 1$, whether $\mu(U_2) = \infty$ or not. 
The contradiction shows that we must have 
$\mu(U) = \infty$.

Now we shall show that if $\mu(U_1), \mu(U_2) < \infty$ then also $ \mu(U) < \infty$.
Suppose to the contrary that $ \mu(U) = \infty$. Pick a natural number $ n > \mu(U_1) + \mu(U_2) + 2$.
Choose compact $C \se U$ such that $\mu(K) > n$ whenever $ K \in \kx, C \se K \se U$.
Also for $ i=1,2$ pick compact sets $C_i \se U_i$ such  that $| \mu(U_i) - \mu(C_i) | < 1$. We may assume that $ C \se C_1 \sc C_2$, 
so $ \mu(C_1 \sc C_2) > n$. But also $ \mu(C_1 \sc \C_2) = \mu(C_1) + \mu(C_2)  \le \mu(U_1 ) + \mu(U_2) + 2 < n$. 
Thus, we must have $ \mu(U) <  \infty$.

We are left to show that $ \mu(U_1) + \mu(U_2) = \mu(U)$  when $ \mu(U_1), \mu(U_2), \mu(U) < \infty$. 
Given $ \eps >0$, we may choose compact sets $K_1, K_2, K$ such that $K \se U, K_i \se U_i,\  K = K_1 \sc K_2$ and 
$ | \mu(U) - \mu(K) | < \eps, | \mu(U_i) - \mu(K_i) | < \eps$ for $i=1,2$. 
Then 
\begin{align*}
\mu(U_1)  + \mu(U_2) &\le \mu(K_1) + \mu(K_2) +  2\eps = \mu(K) + 2 \eps \le \mu(U) + 3 \eps \\
& \le \mu(K) +4 \eps = \mu(K_1 )  + \mu(K_2)  +4 \eps \le \mu(U_1) + \mu(U_2) + 4 \eps.
\end{align*}
Finite additivity on open sets follows.

The fact that finite additivity on open sets implies finite additivity on compact sets can be proved in a similar way.
\end{proof}  

\begin{definition} \label{SDTMnorDe}
For a signed  deficient topological measure $\nu$ we define  $\norm \nu \norm = \sup \{ | \nu(K)|  :  K \in  \kx \} $.
We denote by $\SDTM(X)$ the collection of all signed deficient topological measures on $X$ for which 
 $\norm \nu \norm < \infty$.
\end{definition}

\begin{remark} \label{SDTMcone}
Note that   $\norm \nu \norm =  \sup \{ | \nu(A)|  :  A \in \ox \cup \kx \} = \sup \{ | \nu(U)|  :  U \in \ox \}$.  
The collection of all real-valued signed deficient topological measures is a linear space. 
Any $\nu \in \SDTM(X) $ is real-valued, and $\norm \nu \norm$ is a norm on a linear space $\SDTM(X)$.
\end{remark} 

Studying a signed  deficient topological  measure it is beneficial to consider  its positive, negative and 
total variation, defined in Definition \ref{laplu}. 

\begin{remark} \label{lapldtm1}
Let $\nu$ be a  signed deficient topological measure on a locally compact space $X$. 
By Proposition 21 in Section 3 of \cite{Butler:DTMLC}
the set functions $\nuplu, \numin, |\nu|$ defined in Definition \ref{laplu} 
are deficient topological measures, and 
$ |\nu| \le \nuplu + \numin$.  Also,
$\nuplu $  is the unique 
smallest deficient topological measure such that $ \nuplu \ge \nu$ 
and $\numin $  is the unique largest deficient topological measure 
such that $ - \numin  \le \nu $.
Note that $ \norm \nu \norm < \infty$ if and only if $\nuplu$ and $ \numin$ are finite, i.e. $ \nuplu(X), \numin(X) < \infty$.
To define $\nuplu, \numin$ and $| \nu|$ on $ \ox$  one may use instead of compact sets open sets or sets from $\ox \cup \kx$. 
\end{remark} 

\begin{remark} \label{rmk} 
Let $X$  be locally compact and let $\nu$ be a signed deficient topological measure on $X$. 
From  Lemma 10 in Section 2 of \cite{Butler:DTMLC} we have: (a)  $|\nu(A) | \le |\nu| (A)$ for any $A \in \ox \cup \cx$;
(b) superadditivity:  if $ \bsc_{t \in T} A_t \subseteq A, $  where $A_t, A \in \ox \cup \cx$,  
and at most one of the closed sets is not compact,
then $ \sum_{t \in T } |\nu(A_t)| \le  \sum_{t \in T } |\nu| (A_t) \le | \nu| (A).$ 
\end{remark}
 
\begin{lemma}  \label{SDTMsmall}
Let $X$ be locally compact.The following holds for a signed deficient topological measure $ \nu$:
\begin{enumerate}[label=(d\arabic*),ref=(d\arabic*)]
\item \label{d1}
Given $U \in \ox$ with $ | \nu| (U) < \infty$  and $\eps >0$, there exists $C \se U, \ C \in \kx$ such that 
$ |\nu(A)| \le | \nu | (A) < \eps$  for any compact or open $ A \se U \sm C$. 
\item \label{d2}
Given $C \in \cx$ with $ | \nu| (C) < \infty$ and $\eps >0$, there exists $U \in \ox, \ C \se U$ such that 
$ | \nu(A) | \le | \nu | (A) < \eps$  for any compact or open $ A \se U \sm C$. 
\item \label{d3}
Given $U \in \ox$ with $ | \nu| (U) < \infty$  and $\eps >0$, there exists $C \se U, \ C \in \kx$ such that for any sets 
$A, B \in \ox \cup \kx, \ C \se A \se B \se U$ we have $ | \nu(A) - \nu(B)|  < \eps$. 
\end{enumerate}
\end{lemma}

\begin{proof}
Let $ \eps>0$.   In part \ref{d1} given $U \in \ox$ choose $C \in \kx, C \se U$ and 
in part \ref{d2} given $ C \in \cx$ choose $ U \se \ox, C \se U$ such that $|\nu|(U) - |\nu|(C)  < \eps$.
Then by monotonicity and superadditivity of $ |\nu|$, we have
$$ |\nu(A)| \le | \nu| (A)  \le | \nu| ( U \sm  C)  \le  | \nu| (U) -  | \nu| (C) < \eps.$$

Now we shall show part \ref{d3}. Since $ | \nu| (U) < \infty$, we have $| \nu| (A) < \infty$ so $\nu(A) \in \r$ 
for any $A \se U, A \in \kx \cup \ox$. For $ \eps >0$ we may find 
$ C \in \kx, C \se U$ such that $| \nu(K) - \nu(U) | < \eps/4$ whenever  $K \in \kx, \, C \se K \se U$.
If $A, B \in \kx$ then $ |\mu(A) - \nu(B) | \le | \nu(A) - \nu(U)| + | \nu(U) - \nu(B) | < \eps/2 < \eps$. 
If $A, B \in \ox$ then find compact sets $K, D$ such that $ C \se K \se A,  C \se D \se B$ and 
$ | \nu(A) - \nu(K)| < \eps/4,  | \nu(B) - \nu(D)| < \eps/4$. Then
$ | \nu(A) - \nu(B)| \le | \nu(A) - \nu(K)| + | \nu(K) - \nu(D)| + | \nu(D) - \nu(B) | < 3 \eps/4 < \eps$. 
The remaining two cases can be proved similarly. 
\end{proof}

\noindent
Lemma  \ref{SDTMsmall} helps us to prove the following result.

\begin{lemma} \label{tildenu}
Let $X$ be locally compact. Suppose $\nu$ is a signed deficient topological measure on $X$.
For each open set $U$ define $\widehat \nu(U) =  \sup \{ | \nu(C)| : \ \ C \se U, \ C \in \kx \} $. Then 
$\widehat \nu(U) =  \sup \{ | \nu(V)| : \ \ V \se U, \ V \in \ox \} =  \sup \{ | \nu(A)| : \ \ A \se U, \ A \in \ox \cup \kx \}$, and
$\widehat \nu(U) \le | \nu | (U) \le 2 \widehat \nu(U)$.
\end{lemma}

\begin{proof}
From the definition of a signed deficient topological measure 
we see that $\widehat \nu(U) = \sup \{ | \nu(V)| : \ \ V \se U, \ V \in \ox \} = \sup \{ | \nu(A)| : \ \ A \se U, \ A \in \ox \cup \kx \}$.
From Remark \ref{rmk}, $\widehat \nu(U) \le | \nu | (U)$. 
%We shall show that $ | \nu | (U) \le 2 \widehat \nu(U)$. 
For a finite disjoint collection $\{ K_i :  K_i \in \kx, \  i \in I \}$ let 
$I^+ = \{ i \in I:  \nu(K_i) \ge 0 \}, \ \ \ I^- = \{ i \in I:  \nu(K_i) < 0 \}, \ \ \ K^+ = \bsc_{i \in I^+} K_i$, and 
 $K^- = \bsc_{i \in I^-} K_i$. If   $ K_i  \se U$ for $  i \in I $ then 
%\begin{align*}
$\sum_{i \in I} |\nu(K_i)|  = \sum_{i \in I^+} \nu(K_i) + \sum_{i \in I^-} -\nu(K_i) = \nu(K^+) - \nu(K^-) 
= |\nu(K^+)| + |\nu(K^-) | \le 2 \widehat \nu(U). $
%\end{align*}
%Taking supremum of 
%  $ over all  finite $\{ U_i : \  i \in I \} \se U$ 
%the left hand side, we obtain $
Thus,  $ | \nu | (U) \le 2 \widehat \nu(U)$.
\end{proof}  

\begin{theorem} \label{SDTMban}
The space $\SDTM(X)$ is a normed linear space under either of the two equivalent norms:
$\norm \nu \norm_1 = \sup \{ | \nu(K)|: \ \ \ K \in \kx \}, \ \ \ \norm \nu \norm_2 = \sup \{ | \nu| (K): \ \ \ K \in \kx \} = |\nu|(X)$.
%Is $\DTM$ a closed subspace? If yes, then taking positive variation is a projection from $\SDTM(X)$ to $\DTM(X)$.
\end{theorem}

\begin{proof}
It is easy to see that $\sup \{ | \nu| (K): \ \ \ K \in \kx \} = |\nu|(X)$, and that it is a norm.
From Lemma \ref{tildenu} we see that $\norm \nu \norm_1 \le \norm \nu \norm_2 \le  2\norm \nu \norm_1$, 
so these two norms are equivalent. 
%By Remark \ref{SDTMcone} $\SDTM(X)$ is a normed linear space.
%The space $\STM(X)$ is a Banach space see Daniel's paper "Signed q.m.".
\end{proof}

\begin{theorem} \label{STMsmoo}
Suppose $ \nu$ is a signed deficient topological measure such that at most one of $ \nuplu(X), \numin(X) $ is infinity, or 
$ \nu$ is real-valued.
Then $ \nu$ is $ \tau-$ smooth on open sets, and also $ \tau-$ smooth on compact sets. 
\end{theorem}

\begin{proof}
Suppose $ \nu$ is a signed deficient topological measure such that at most one of $ \nuplu(X), \numin(X) $ is infinity.
(The case where $\nu$ is real-valued is similar but simpler.)
Without loss of generality, let $ \numin(X) \le M$, where $ M \in \N$. 
First we shall show that $ \nu$ is $\tau$-smooth on open sets. 
Supose $U_s \nearrow U, U_s, U \in \ox, s \in S$. 
\begin{enumerate}[label=(\roman*),ref=(\roman*)]
\item \label{smii}
Assume first that $ \nu(U) < \infty$.
Let $ \eps >0$. There exists $K \in \kx, K \se U$ such that $| \nu(D) - \nu(U) | < \eps$  whenever  $D \in \kx, K \se D \se U$. 
By Remark \ref{netsSETS} let $t \in S$ be such that $ K \se U_s$ for any $s \ge t$.
We claim that $\nu(U_s) \in \r$ for each $s \ge t$. (Indeed,  let $s \ge t, \eps$ as above. 
By assumption, we can not have $\nu(U_s) = - \infty$.  Suppose 
that $ \nu(U_s) = \infty$. For $ n > | \nu(U)| + \eps$ pick compact $K_s \se U_s$ such that $\nu(C) > n$ for any $ C \in \kx$ satisfying
$ K_s \se C \se U_s$. Then for the compact $D= K \cup K_s \se U$ we have $ | \nu(D) - \nu(U)| < \eps$ and 
$\nu(D) > n  >  | \nu(U)| + \eps \ge \nu(U) + \eps$, which gives a contradiction).   

For each $s \ge t$, pick $C_s \se U_s, C_s \in \kx$ such that 
$ | \nu(U_s) - \nu(C) | < \eps$ for any compact $C$ satisfying $C_s \se C \se U_s$. Then
$$ | \nu(U) - \nu(U_s)| \le | \nu(U) - \nu(K \cup C_s) | +| \nu(K \cup C_s)) - \nu(U_s) |  < 2 \eps,$$
and $ \nu(U_s) \rightarrow \nu(U)$.
\item
Now assume that $ \nu(U) = \infty$. For $ n > M$  pick $K \in \kx$ such that $ K \se U$ and $\nu(D) > 2n$ 
whenever $D \in \kx,K \se D \se U$.
By Remark \ref{netsSETS} let $t \in S$ be such that $ K \se U_s$ for any $s \ge t$. 
%For any such $s$, either $\nu(U_s) = \infty > n $ or 
Suppose $\nu(U_s) < \infty$. 
For $0 < \eps <n$,  there exists $D_s \in \kx, D_s \se U_s$ such that $| \nu(U_s) - \nu(D) | < \eps$ 
for any $ D \in \kx, D_s \se D \se U_s$. Then 
$$ | \nu(U_s) | \ge  | |\nu(U_s) - \nu(D_s \cup K)| - | -\nu(D_s \cup K) | | \ge 2n - \eps > n.$$
It follows that for any $ s \ge t$, whether $\nu(U_s) < \infty $ or $\nu(U_s) = \infty$, we have $ | \nu(U_s)| >n > M \ge \numin(X)$, 
so $\nu(U_s)$ can not be negative. Thus, for $ s \ge t$ we have  $  \nu(U_s) =| \nu(U_s)|  > n$, so $ \nu(U_s) \rightarrow \infty$.
\end{enumerate}
Thus, $ \nu$ is  $\tau$-smooth on open sets. We may show that $ \nu$ is $\tau$-smooth on compact sets in a similar fashion. 
\end{proof}

\section{Signed topological measures on a locally compact space}

\begin{definition} \label{STMLC}
A signed topological measure on a locally compact space $X$ is a set function
$\mu: \ox \cup \cx \rightarrow [-\infty, \infty]$  that assumes at most one of $\infty, 
-\infty$ and satisfies the following conditions:
\begin{enumerate}[label=(STM\arabic*),ref=(STM\arabic*)]
\item \label{STM1} 
if $A,B, A \sc B \in \kx \cup \ox $ then
$\mu(A\sqcup B)=\mu(A)+\mu(B);$
\item \label{STM2}  
$\mu(U)=\lim\{\mu(K):K \in \bcx, \  K \se U\}
$ for $U\in\ox$;
\item \label{STM3}
$\mu(F)=\lim\{\mu(U):U \in \ox, \ F \se U\}$ for  $F \in \cx$.
\end{enumerate}
By $STM(X)$ we denote the collection of all signed topological measures on $X$. 
\end{definition} 

\begin{lemma} \label{EqSTMC}
Suppose $ \mu: \ox \cup \cx \rightarrow [ - \infty, \infty ] $ is a set function that assumes at most one of $\infty, 
-\infty$ and that satisfies the following conditions:
\begin{enumerate}[label=(\roman*),ref=(\roman*)]
\item \label{STM2e} 
$ \mu(U) = \lim\{ \nu(K) : \ K \se U, \ K \in \kx \} $ for  $ U \in \ox$;
\item \label{STM3e} 
$ \mu(F) = \lim\{ \mu(U) : \ F \se U, \ U \in \ox \} $ for $F \in \cx$; 
\item \label{STM1e}
$\mu$ is finitely additive on $\ox$ or $ \nu$ is finitely additive on $ \kx$;
\item \label{STM1ed}
if $K \sc V = W, \ K \in \kx, \ V, W \in \ox$ then $ \mu(K)  + \mu(V) = \mu(W)$;
\item \label{dopYs}
for each $K \in \kx$ there exists  an open neighborhood $W$ of $K$ such that $ \mu(W \sm K) \in \r$. 
\end{enumerate}
Then $ \mu$ is a signed topological measure on $X$. 
In particular, any real-valued set function $ \mu$ on  $  \ox \cup \cx$ that satisfies \ref{STM2e} - \ref{STM1ed} is a
real-valued signed topological measure. 
\end{lemma}

\begin{proof}
We need to check the condition \ref{STM1} of Definition \ref{STMLC}.
By Lemma \ref{finAddOC} $\mu$ is finitely additive on $ \ox$ and on $ \kx$, so we only need to show that 
if $K \sc V = C$ where $ K , C \in \kx,  \, V \in \ox$ then $ \mu(K) + \mu(V) = \mu(C)$.
Let $W \in \ox$ be such that $ C \se W$ and $ \mu(W \sm C) \in \r$. 
Then $K \sc V  = C \se W$, so $K \sc V \sc (W \sm C) = W = C \sc (W \sm C)$. 
By part \ref{STM1ed} and finite additivity of $ \mu$ on open sets 
$$ \mu(K) + \mu(V) + \mu(W \sm C) = \mu(W) = \mu(C) + \mu(W \sm C),$$ 
so $ \mu(K) + \mu(V) = \mu(C)$,
and the statement is proved.
\end{proof}

When $X$ is compact Definition \ref{STMLC} simplifies to the following:

\begin{definition} \label{STMC} 
A signed real-valued topological measure on a compact space $X$ is a set function
$\mu: \ox \cup \cx \longrightarrow (-\infty, \infty)$  that satisfies the following conditions:
%\begin{enumerate}[label=(t\arabic*),ref=(t\arabic*)]
%\item \label{t1}
\begin{enumerate}[label=(c\arabic*),ref=(c\arabic*)]
\item \label{STMC1} 
if $A,B, A \sc B \in \kx \cup \ox $ then
$\mu(A\sqcup B)=\mu(A)+\mu(B);$
\item \label{STMC2}  
$\mu(U)=\lim\{\mu(K):K \in \bcx, \  K \se U\}$ for $U\in\ox$;
\end{enumerate}
\end{definition}

\begin{remark} \label{eqCoreva}
Condition \ref{STMC2} of Definition \ref{STMC}  is equivalent to: 
$$ \mu(F) = \lim\{ \mu(U) : \ F \se U, \ U \in \ox \}  \mbox{   for   }  F \in \cx.$$

As was noticed in \cite{Grubb:Signed}, condition \ref{STMC1} of Definition \ref{STMC} is equivalent to the following three conditions:
\begin{enumerate}[label=(\roman*),ref=(roman*)]
\item
$ \mu(U \sc V) = \mu(U) + \mu(V)$ for any two disjoint  open sets $U,V$.
\item
If $X= U \cup V, \ \ \ U , V \in \ox$ then $ \mu(U) + \mu(V) = \mu(X) + \mu(U \cap V)$.
\item
$\mu(X \sm U) = \mu(X) - \mu(U)$ for any open set $U$.
\end{enumerate}
Thus, when $X$ is compact,  a real-valued signed topological measure can be defined by its actions on open sets. 
\end{remark}

\begin{lemma} \label{sdtm=tm}
Suppose $ \mu$ is a signed deficient topological measure, $ \nu$ is a signed topological measure 
(respectively, a deficient topological measure), and $\mu = \nu$ on $\kx$.
Then $ \mu = \nu$ and $ \mu$ is a  signed topological measure (respectively, a deficient topological measure).
\end{lemma}

\begin{proof}
If $ \mu$ is a signed deficient topological measure, $ \nu$ is a signed topological measure, and $\mu = \nu$ on $\kx$, then also
$ \mu = \nu$ on $ \ox$. It is easy to check condition \ref{STM1} of Definition \ref{STMLC}
(for example,  $ \mu(U \sm K) + \mu(K) = \nu(U \sm K)  + \nu(K) = \nu(U) = \mu(U)$), and $ \mu$ is 
a signed topological measure. Similarly, $\mu$ is a deficient topological measure if $ \nu$ is.
\end{proof} 

\begin{lemma}  \label{STMnets}
Let $X$ be locally compact. Let $ \mu: \ox \cup \cx$ be a real-valued signed set function such that 
$\mu(U) = \mu(K) + \mu(U \sm K)$ 
for any 
$U \in \ox$ and any compact $ K \se U$.
Consider the following conditions: 
\begin{enumerate}[label=(p\arabic*),ref=(p\arabic*)]
\item \label{1}
For any  open set $U$, the limit of the net $\{ \mu(K)\}$ with index set $ \{ K \in \kx: \ K \se U \} $ ordered by inclusion
exists and 
$$ \mu(U) = \lim_{K \se U, \ K \in \kx} \mu(K).  $$
\item \label{2}
Given $U \in \ox$ and $\eps >0$, there exists $ K \in \kx, K \se U$  such that $ |\mu( E)|  < \eps$  for any $ E \se U \sm K, \ E \in \kx \cup \ox$.
\item \label{3}
Given $U \in \ox$ and $\eps >0$, there exists $ K \in \kx, K \se U$ such that $ |\mu( V)|  < \eps$  for any $ V \se U \sm K, \ V \in \ox$.
\item \label{4}
Given $U \in \ox$ and $\eps >0$, there exists $ K \in \kx, K \se U$ such that $ |\mu( C)|  < \eps$  for any $ C \se U \sm K, \ C \in \kx$.
\item \label{5}
For any  compact $K$, the limit of the net $\{ \mu(U)\}$ with index set $ \{ U \in \ox: \ K \se U \} $ ordered by reverse
inclusion exists and 
$$ \mu(K) = \lim_{U \supset K, \ U \in \ox} \mu(U). $$
\item \label{6}
Given $F \in \cx$ and $\eps >0$, there exists $U \supset F,  \ U \in \ox$ such that $ |\mu( C)|  < \eps$  for any $ C \se U \sm F, \ C \in \kx $.
\item \label{7}
Given $F \in \cx$ and $\eps >0$, there exists $U \supset F,  \ U \in \ox$ such that $ |\mu( V)|  < \eps$  for any $ V \se U \sm F, \ V \in \ox $.
\item \label{8}
Given $F \in \cx$ and $\eps >0$, there exists $U \supset F,  \ U \in \ox$ such that $ |\mu(E)|  < \eps$ for any 
$E \se U \sm F, \ E \in \kx \cup \ox$.
\item \label{9}
For any  closed set $F$, the limit of the net $\{ \mu(U) \}$ with index set $ \{ U \in \ox: \ F \se U \} $ ordered by reverse
inclusion exists and 
$$ \mu(F) = \lim_{U \supset F, \ U \in \ox} \mu(U). $$
\end{enumerate}
Then \ref{1},  \ref{2}, and  \ref{3} are equivalent and imply \ref{4}  and  \ref{5}.
If $X$ is compact,  then  \ref{1},  \ref{2}, \ref{3},  \ref{5}, \ref{7},  \ref{8}, \ref{9} are equivalent 
and imply equivalent  conditions  \ref{4}, \ref{6}.
\end{lemma}

\begin{proof}
\ref{1} $ \Longrightarrow $ \ref{2}: Let $ U \in \ox, \eps>0$. By \ref{1} pick $ K  \in \kx$ such that  
$| \mu(U \sm C)| =  |\mu(U) - \mu(C)|  < \eps$ for any $C \in \kx$ satisfying $ K \se C \se U$.
For any compact $ D \se U \sm K$ we have $\mu(D) = \mu(U \sm K) - \mu(( U \sm K) \sm D)$ and so 
$ | \mu(D)|  < 2 \eps.$ For any open set $ V \se U \sm K$  we then see that $ |\mu(V)| \le 2 \eps$, 
and so \ref{2} follows. 
Obviously, \ref{2} $ \Longrightarrow $ \ref{3}. \ref{3} $ \Longrightarrow $ \ref{1}: For $U \in \ox$ and $ \eps >0$ pick $K  \in \kx$
such that $ |\mu( V)|  < \eps$  for any open set $ V \se U \sm K$. Then for any compact $C$ such that $ K \se C \se U$ we have
$ U \sm C \se U \sm K$, so $ | \mu(U) - \mu(C)| = | \mu(U \sm C)| < \eps$, and  \ref{1} follows. 
Thus \ref{1},  \ref{2}, and \ref{3} are equivalent. 
Obviously, \ref{2} $ \Longrightarrow $\ref{4}. \ref{3} $ \Longrightarrow $ \ref{5}: Let $ K \in \kx, \eps >0$. By \ref{3} for the set $U = X \sm K$ 
find compact $ C \se U$ such that $ | \mu(V)| < \eps$ for any open set $ V \se U \sm C$. Setting $W_{\eps} = X \sm C$ we see that 
$ K \se W_{\eps}$ and for any open $W$ such that $K \se W \se W_{\eps}$ we  have 
$ W \sm K \se W_{\eps} \sm K  = (X \sm C) \sm K = (X \sm K) \sm C = U \sm C$, so 
$ | \mu(W) - \mu(K) | = | \mu(W \sm K)| < \eps$, which shows  \ref{5}.  

Suppose $X$ is compact.  \ref{1},  \ref{2}, and \ref{3} are equivalent and imply \ref{5}, which is equivalent to \ref{9}.
From the duality between open and closed sets it is easy to see that \ref{9}  $ \Longrightarrow $ \ref{1},  
\ref{2} $ \Longleftrightarrow $  \ref{8},  \ref{3} $ \Longleftrightarrow$ \ref{7}, and  
\ref{4}  $\Longleftrightarrow$  \ref{6}. This finishes the proof.
\end{proof}

\begin{remark} 
Condition \ref{STMC2} of Definition \ref{STMC} and 
condition \ref{STM2} of Definition \ref{STMLC} (for a real-valued signed topological measure)
could be replaced by any of the equivalent conditions in Lemma \ref{STMnets}.
\end{remark}

We denote by $\STM(X)$ the collection of all signed topological measures on $X$ for which $\norm \mu \norm < \infty$. 
As in Theorem \ref{SDTMban}, we have:

\begin{theorem} \label{STMbana}
The space $\STM(X)$ is a normed linear space under either of the two equivalent norms:
$\norm \mu \norm_1 = \sup \{ | \mu(A)|: \ \ \ A \in \ax \}, \ \ \ \norm \mu \norm_2 = \sup \{ | \mu| (A): \ \ \ A \in \ax \} = |\mu|(X)$.
\end{theorem}

\begin{remark}
Signed topological measures of finite norm on a compact space were introduced in  \cite{Grubb:Signed} 
then studied or used in \cite{Grubb:SignedqmDimTheory}, 
\cite{OrjanAlf:CostrPropQlf},  \cite{Svistula:Signed}, \cite{Svistula:Convergence}. 
In these papers different definitions of a signed topological measure
were given, but their equivalence, as well as equivalence to our Definition \ref{STMC}, follows from Lemma \ref{STMnets}
and Theorem \ref{STMbana}.
\end{remark} 
 
\begin{remark}
Since any signed topological measure is a signed deficient topological measure, 
we may use the definitions and results from section \ref{seSDTM}. 
Note that  in general, when $\mu$ is a signed topological measure,  
$\muplu, \mumin$ and $|\mu|$ are deficient topological measures, but not topological measures.
See, for instance, Example 25 in \cite{OrjanAlf:CostrPropQlf}. 
It is easy to see that if $\mu$ is a signed measure then $\muplu, \mumin$ and $| \mu|$ are the classical
positive, negative, and total variations of a signed measure.
\end{remark}

\begin{lemma}  \label{SDTMnSTMex}
Suppose $X$ is locally compact, $\nu$ is a deficient topological measure on $X$, which is not a 
topological measure on $X$,  $\mu$ is a topological measure on $X$, and $ \norm \nu \norm , \norm \mu \norm < \infty$.
Then $\la = \nu - \mu$ is a signed deficient topological measure on $X$
which is not a signed topological measure. 
\end{lemma}

\begin{proof}
By Remark \ref{SDTMcone} $\la$ is a signed deficient topological measure on $X$. 
Since $ \nu$ is not a topological measure, by Theorem 29 in Section 4 of \cite{Butler:DTMLC} 
there exist open $U$ and compact $C \se U$ such that 
$\nu(U) > \nu(C) + \nu(U \sm C)$.  Since $\mu$ is a topological measure, 
$\mu(U) = \mu(C) + \mu(U \sm C)$. 
Then 
\begin{align*}
\la(U) &= \nu(U) -\mu(U) = \nu(U) - \mu(C) -\mu(U \sm C) \\
&> \nu(C) + \nu(U \sm C) -\mu(C) -\mu(U \sm C) = \la(C) + \la (U \sm C),
\end{align*}
which shows that $\la$ is not a signed topological measure.
\end{proof}

The next few results allow us to use a smaller collection than $\kx$ to check the equality of two signed topological measures.
A set $A \subseteq X$ is called bounded if $\cl A$ is compact. 
A set $A$ is called solid if $A$ is  connected, and $X \sm A$ has only unbounded connected components.
Let $\bcox, \ \bccx, \ \bcsx$  and $ \bosx$ denote, respectively,  
the family of finite unions of disjoint compact connected sets, the family of compact connected sets,  the family of compact solid sets,
and the family of bounded open solid sets.

When $X$ is compact, a set is called solid if it and its complement are both connected.
For a compact space $X$ we define a certain topological characteristic, genus. 
See \cite{Aarnes:ConstructionPaper} for more information about genus $g$ of the space. 
A compact space has genus 0 iff any finite union of disjoint closed solid sets has a connected complement.
%Another way to describe the ``$g=0$'' condition is the following: if the union of two open
%solid sets in $X$ is the whole space, their intersection must be connected. (See \cite{Grubb:IrrPart}.) 
Intuitively, $X$ does not have holes or loops.
In the case where $X$ is locally path connected, $g=0$ if the fundamental group $\pi_1(X)$ is finite (in particular, if $X$ is 
simply connected). Knudsen \cite{Knudsen} was able to show that if 
$H^1(X) = 0 $ then $g(X) = 0$, and in the case of CW-complexes the converse also holds.

\begin{example} \label{STMnotSM}
Let $X = \r^2$.
Consider $ \mu = \nu - \frac12 \delta$, where $\nu$ is a simple topological measure which is not a measure on $X$,
and $ \delta$ is a point mass.  
(As $\nu$ one may take, for instance, a topological  measure from Example 1 or Example 2 in  \cite{Butler:TechniqLC}). 
Then $ \mu$ is a signed topological measure. We shall show that $ \mu$ is not a signed measure by 
demonstrating that $ \mu^+$ is not subadditive.
By Theorem 34 in Section 4 of \cite{Butler:DTMLC}, there are open sets 
$U, V$ such that $ \nu(U \cup V) > \nu(U) + \nu(V)$. Since $\nu$ is simple, we have $ \nu(U \cup V) = 1, \nu(U) = \nu(V) = 0$.
Pick also a compact $ K \se U \cup V$ such that $ \nu(K) = 1$. Then $ \mu^+ (U \cup V) \ge \mu(K) \ge 1 - \frac12 = \frac12$.
Since $ \mu^+(U) = \sup\{ \mu(C): \, C \in \kx, C \se U\} \le  \sup\{ \nu(C): \, C \in \kx, C \se U\}  = \nu(U) =0$, 
we have $\mu^+(U) =0$. Also, $ \mu^+(V) = 0$, and we see that $ \mu^+$ is not subadditive.
\end{example}

%\begin{example} \label{negSet}
%Let $X = \r^2$, and a signed topological measure $ \mu = \nu - \frac12 \delta$, where $ \nu$ is a simple topological measure on $X$ 
%given by Example 1 or Example 2 in \cite{Butler:TechniqLC}), and $ \delta $ is a point mass at a point $y \in X$. 
%(Note that $y$ may or may not coincide with the point used to construct $ \nu$.) 
%Let $E = \{ y\}$, a compact solid set.  
%Then $ \mu(E) = - \frac12$.
%\end{example}

\begin{remark} \label{Vloz}
Let $X$ be locally compact. 
We denote by $M(X)$ the collection of all Borel measures on $X$ that are inner regular on open sets and outer regular 
(restricted to $\ox \cup \cx$). Thus, $M(X)$ includes regular Borel measures and Radon measures. Let $SM(X)$ be the family of signed measures 
that are differences of two measures from $M(X)$, one of which is finite. 
We have:
\begin{align} \label{incluMTD}
 M(X) \subsetneqq  TM(X)  \subsetneqq  DTM(X)
\end{align}
and 
\begin{align} \label{incluSMTD}
 SM(X) \subsetneqq  STM(X)  \subsetneqq  SDTM(X).
\end{align}

The inclusions follow from the definitions. Inclusions in (\ref{incluSMTD}) are proper by Lemma \ref{SDTMnSTMex} and Example \ref{STMnotSM}.
When $X$ is compact, there are examples of topological measures that are not measures 
and of deficient topological measures that are not topological measures in numerous papers, 
beginning with \cite{Aarnes:TheFirstPaper}, \cite{OrjanAlf:CostrPropQlf}, and  \cite {Svistula:Signed}.
When $X$ is locally compact, see \cite{Butler:TechniqLC},
Sections 5 and 6 in \cite{Butler:DTMLC}, and  Section 9 in \cite{Butler:TMLCconstr} 
for more information on proper inclusion in (\ref{incluMTD}),  criteria for a topological measure to be a measure from $ \M(X)$, and various examples.
\end{remark}

\begin{remark} \label{STMlimCo}
As in Remark 14 in Section 3 of \cite{Butler:DTMLC}, we have the following.
Let $\nu$ be a  signed deficient topological measure on $X$.
If $X$ is locally compact and locally connected then for each open set $U$
\[  \nu(U) = \lim\{ \nu(K) : \ K \subseteq U , \ K \in \bcox \}. \]
If $X$ is  locally compact, connected, and locally connected, then 
\[  \nu(X)  = \lim\{ \nu(K) : \  K \in \bccx \}, \] 
and also 
\[  \nu(X)  = \lim\{ \nu(K) : \  K \in \bcsx \}. \] 
\end{remark}

\begin{lemma} \label{STMlimSol}
Let $\nu$ be a signed topological measure on a locally compact, locally connected space $X$.
If $U \in \bosx$ then  $ \nu(U) = \lim\{ \nu(C) : \ C \se U, \ C \in \bcsx \}$.
\end{lemma}

\begin{proof}
Follows from Lemma 21 in Section 3 of \cite{Butler:TMLCconstr}.
\end{proof}

\begin{theorem} \label{STMSolid}
Let $X$ be a locally compact, connected, and locally connected space.
If $\nu$ and $\mu$ are real-valued signed topological measures on  $X$  
such that $\mu = \nu$ on compact solid sets then $\nu = \mu$.
\end{theorem}

\begin{proof}
Suppose that $\nu(K) = \mu(K)$ for every $K \in \bcsx$. 
By Remark \ref{STMlimCo} $ \nu(X) = \mu(X)$. 
By Lemma \ref{STMlimSol} $\nu(V) = \mu(V)$ for every $V \in \bosx$.  
For a compact connected set $C$ by Lemma 18 in Section 3 of \cite{Butler:TMLCconstr} we have: 
\begin{eqnarray} \label{uravSolHu}
X = \tilde C \sc \bsc_{i=1}^n U_i = C \sc \bsc _{s \in S} V_s  \sc \bsc_{i=1}^n U_i, 
\end{eqnarray}
where $U_i$ are unbounded connected components of $X \sm C$, $V_s$ are bounded connected components of $X \sm C$, 
and $\tilde C$ is a solid hull of $C$.  By Lemma 20 in Section 3 of \cite{Butler:TMLCconstr}, $\tilde C$ is a compact solid set. 
Since  $\mu = \nu$ on compact solid sets and $\nu(X) = \mu(X)$, from the first equality in (\ref{uravSolHu})
by finite additivity of signed topological measures on $\ox \cup \kx$
we see that $\nu(\bsc_{i=1}^n U_i ) = \mu(\bsc_{i=1}^n U_i)$.   
By Lemma 17 in Section 3 of \cite{Butler:TMLCconstr} each $V_s \in \bosx$. 
By Lemma \ref{finAddOC} both $ \mu$ and $ \nu$ are additive on open sets, so 
from  (\ref{uravSolHu}) we have: $\nu(C) = \mu(C)$ 
for each $C \in \bccx$. Then $\nu = \mu$ on $\bcox$. Then by Lemma \ref{STMlimCo} $\nu = \mu$ on $\ox$.
Thus, $\nu = \mu$.
\end{proof} 

\begin{remark}
Theorem \ref{STMSolid} suggests the possibility of obtaining a signed topological measure as the unique extension 
of a suitable set function defined only on bounded solid sets with methods similar to ones in  \cite{Butler:TMLCconstr} 
(where on a locally compact, connected, locally connected space a solid-set function is extended uniquely to a topological measure) 
and in section 8 in \cite{OrjanAlf:CostrPropQlf} 
(where on a connected, locally connected, compact Hausdorff space a signed solid-set function is extended uniquely to a finite signed topological measure).  
\end{remark} 

\section{Decomposition of signed topological measures into deficient topological measures}
  
\begin{lemma} \label{sdtmDEC}
Let $X$ be locally compact.
Let $\la: \kx \cup \ox \rightarrow [- \infty, \infty] $ be a signed set function that assumes at most one of $\infty,  - \infty$. 
Suppose $\la$ is finitely additive on $\kx$ and 
$\la(U) = \la(U \sm K )  + \la(K)$ for any open $U$ and any compact $K \se U$ .
Let $\laplu, \lamin$ and $| \la|$ be as in Definition \ref{laplu}.
Then 
\begin{enumerate}[label=(\Roman*),ref=(\Roman*)]
\item \label{firpaU}
$\la (U)  = \laplu (U) - \lamin (U)$ for any open set $U$ such that at least one of $\laplu(U), \lamin(U)$ is finite.
\item
$| \la |  = \laplu + \lamin$.
\end{enumerate}
\end{lemma}

\begin{proof}
\begin{enumerate}[label=(\Roman*),ref=(\Roman*)]
\item 
Suppose $ \la$ does not assume $ - \infty$. Let $ U \in \ox$.
For any compact $K \se U$ we have
\begin{align} \label{laminKE}
\la(U) = \la(U\sm K) + \la(K) \ge - \lamin(U \sm K) + \la(K) \ge - \lamin(U) + \la(K)
\end{align}

Assume that  $ \lamin(U) < \infty$. There are two possibilities for $ \laplu(U)$.
Suppose that $ \laplu(U) <\infty$.
Given $\eps >0$,  we choose a compact set $ K \se U $ such that $ \la(K) > \laplu(U) - \eps$. Then from (\ref{laminKE})
\[ \la(U)  \ge - \lamin(U) + \la(K)  \ge -\lamin(U)  + \laplu(U) -\eps, \]
so $\la(U) \ge  \laplu(U) - \lamin(U)$. 
Since $\lamin(U) <  \infty$, we may apply the same argument to $-\la$ to get: 
$-\la(U) \ge ( -\la)^+ (U) - ((-\la)^-(U)) = \lamin(U)  - \laplu(U)$, 
i.e.  $\la(U) \le \laplu(U) - \lamin(U)$.
Therefore,   $\la(U) =  \laplu(U) - \lamin(U)$. 

Now suppose $ \laplu(U) = \infty$. For a natural number $n$ choose $ K \se U$ such that $\la(K) > n$. Then from  (\ref{laminKE})
we have
\[ \la(U)  \ge - \lamin(U) + \la(K) > n - \lamin(U). \]
Letting $ n \rightarrow \infty$ we obtain $ \la(U) = \infty$. We again have $\la(U) =  \laplu(U) - \lamin(U)$. 

\item
To prove that $| \la| = \laplu + \lamin$, by Lemma 10 in Section 2 of \cite{Butler:DTMLC} we only need 
to show that $| \la| \ge \laplu + \lamin$, and it is enough to check this inequality on open sets.
Let $U$ be open. By Lemma 10 in Section 2 of \cite{Butler:DTMLC} the equality is trivial if $\laplu(U) = 0$ or $\lamin(U) = 0$. 
So we assume that $\laplu(U) > 0$ and $\lamin(U) > 0$. 
If $\laplu(U) = \infty$ or $ \lamin(U) = \infty$ the statement holds because $|\la| \ge \laplu, |\la| \ge \lamin $.
Now assume that $0< \laplu(U), \lamin(U) < \infty$.
Given $\eps>0$, choose a compact $K$ such that $ \laplu(U) - \la(K) < \eps$ and $\la(K) >0$.
Note that $ \la(U\sm K) + \lamin(U) = \la(U) -\la(K) + \lamin(U) = \laplu(U) - \la(K) < \eps$, i.e.
$ - \la(U \sm K) > \lamin(U) - \eps$.
By superadditivity of $ | \la|$
\begin{align*}
| \la | (U) &\ge | \la | ( U \sm K) + | \la | (K) \\  
 & \ge  | \la(U\sm K)| + | \la(K) |  = | \la(U\sm K)|  + \la(K)  \\
& \ge - \la(U \sm K) + \la(K)  \ge   \lamin(U) - \eps + \laplu(U) - \eps,
\end{align*}  
i.e. $|\la| (U) \ge \laplu(U) + \lamin(U)$. Thus,  $| \la| \ge \laplu + \lamin$.
\end{enumerate}
\end{proof}

\begin{remark}
Lemma \ref{sdtmDEC} is related to parts 7 and 8 of Proposition 24 in \cite{OrjanAlf:CostrPropQlf}.
\end{remark}

\begin{corollary} \label{finiPMA}
Suppose $ \la$ is a signed topological measure, $A \in \ox \cup \cx$. 
\begin{enumerate}[label=(\roman*),ref=(\roman*)]
\item \label{obpa1} 
$ |\la(A)| < \infty$ implies $ \laplu(A), \lamin(A)$ are both finite or both infinite.
\item  \label{obpa2}
If $ \la$ is such that at least one of $ \laplu(X), \lamin(X)$ is finite, then
$ |\la(A)| < \infty$ if and only if $ \laplu(A), \lamin(A) < \infty$.
\end{enumerate}
\end{corollary}

\begin{proof}
\begin{enumerate}[label=(\roman*),ref=(\roman*)]
\item 
Let $U$ be open, $ | \la(U) | < \infty$. 
If exactly one of $\laplu(U), \lamin(U)$ were finite, it would contradict part \ref{firpaU} of Lemma \ref{sdtmDEC}. 
So the corollary holds for open sets. 
Suppose $F \in \cx$ and $ \la(F) \in \r$. There is an open set $U$ containing $F$ for which
$ \la(V) \in \r$ for any $V \in \ox$ with $ F \se V \se U$. If there is $V$ with $\laplu(V), \lamin(V) < \infty$ then 
$ \laplu(F) \le \laplu(V) < \infty$, and also $ \lamin(F) < \infty$. If for all  $V$ both $ \laplu(V), \lamin(V)$ are infinite, then 
also $ \laplu(F),\lamin(F)$ are infinite.
\item
Since at least one of $ \laplu, \lamin$ is finite, the direction ($\Longrightarrow$)  follows from part \ref{obpa1}. 
The direction ($ \Longleftarrow$) for an open set follows from part \ref{firpaU} of Lemma \ref{sdtmDEC}, 
and then is easily checked for a closed set. 
\end{enumerate}
\end{proof} 

\begin{theorem} \label{nupluTh}
Suppose $ \la $ is a signed topological measure.
\begin{enumerate}[label=(\Roman*),ref=(\Roman*)]
\item
The positive variation $\laplu $  is the unique 
smallest deficient topological measure such that $ \laplu \ge \la$, 
and  the negative variation $\lamin $  is the unique largest deficient topological measure 
such that $ - \lamin \le \la$; also, $ |\la| = \laplu + \lamin$. 
\item
If at least one of $ \laplu, \lamin$ is finite (in particular, if $ \norm \la \norm < \infty$) then also $ \la = \laplu - \lamin$.
\end{enumerate}
\end{theorem}

\begin{proof}
\begin{enumerate}[label=(\Roman*),ref=(\Roman*)]
\item
Follows from Remark \ref{lapldtm1} and Lemma \ref{sdtmDEC}.
\item
Assume that at least one of $ \laplu, \lamin$ is finite. We shall show that $ \la = \laplu - \lamin$.
Without loss of generality we may assume that $\la$  does not assume $-\infty$.
Note that $ \lamin(X) < \infty$, for otherwise we would have $ \laplu(X) < \infty$, 
and by part \ref{firpaU} of Lemma \ref{sdtmDEC} $ \la(X) = -\infty$. 
Assume that $ \laplu(X) = \infty$ (The case $ \laplu(X) < \infty$ is similar but simpler).
By part \ref{firpaU} of Lemma \ref{sdtmDEC} the equality 
$ \la = \laplu - \lamin$ holds for open sets. Let $F \in \cx$.  We have $ \lamin(F) < \lamin(X) < \infty$.
If $ \la(F) = \infty$ then from part \ref{obpa2}  of Corollary \ref{finiPMA} we see that $ \laplu(F) = \infty$. Then 
$ \la(F) = \laplu(F) - \lamin(F)$. Now suppose $  \la(F) \in \r$. There exists $W \in \ox$ such that $ \la(U) \in \r$ for all 
$ U \in \ox, \, F \se U \se W$. 
By part \ref{obpa2} of Corollary \ref{finiPMA} $ \laplu(F), \lamin(F), \laplu(U), \lamin(U) \in \r$.
Then $\la(F) =  \lim \la(U)  =  \lim ( \laplu(U) - \laplu(U)) = \lim \laplu(U) - \lim \lamin(U) =\laplu(F)  - \lamin(F)$.
\end{enumerate}
\end{proof}

\section{Decomposition of signed topological measures into topological measures}

The following is Theorem 5 in \cite{Butler:TechniqLC}. 

\begin{theorem}  \label{DTMnupLC}
Let $X$ be a locally compact, connected, locally connected space whose one-point compactification has genus 0. 
Let $\nu$ be a deficient topological measure on $X$ such that $\nu(X) < \infty$ and let $p \in X$ be an arbitrary point. 
Define a set function $\nu_p : \bosx \cup \ksx \rightarrow [0, \infty)$ by
\begin{eqnarray*}
  \nu_p(A) & = &
  \left\{
  \begin{array}{rl}
  \nu(A), & \mbox{ if }  p \notin A \\
  \nu(X) - \nu(X \sm A) , & \mbox{ if } p \in A
  \end{array}
  \right.
\end{eqnarray*} 
Then 
$\nu_p$ is a solid set function and,  hence, extends to a topological measure on $X$. 
\end{theorem} 

\begin{theorem} \label{STMdecom}
Suppose $X$ is a connected, locally connected, locally compact (non-compact) space whose one-point compactification has genus 0. 
Let $\nu$ be a signed topological measure with finite norm on $X$.  
Then $\nu$ can be represented as a difference of two topological measures.
\end{theorem}

\begin{proof}
By Theorem \ref{nupluTh}, $\nu$ can be represented as the difference of two deficient topological measures,
$\nu = \nuplu - \numin$. By Remark \ref{lapldtm1}, $\nuplu(X) < \infty$ and $\numin(X) < \infty$.
Let $p \in X$. From $ \nuplu, \numin$ by Theorem \ref{DTMnupLC} obtain $\nu_1 = \nuplu_p$ and $\nu_2 = \numin_p$. 
Then $\nu_1, \nu_2$ are topological measures, and  $\mu = \nu_1 - \nu_2$  is a signed topological measure. 
We shall show that $\mu = \nu$.
If $ A $ is a bounded open solid set or a compact solid set and $p \notin A$, then using Theorem \ref{nupluTh}, we have:
$$ \mu(A) = \nu_1(A) - \nu_2(A) = \nuplu(A)  - \numin(A) = \nu(A).$$

Now let  $K$ be a compact solid set, $p \in K$. 
Since $ \nu_1(K) = \nuplu_p(K) = \nuplu(X) - \nuplu(X \sm K)$ and $ \nu_2(K) = \numin(X) - \numin(X \sm K)$, 
by Theorem \ref{nupluTh} we have:
\begin{align*}
\mu(K) =\nu_1(K) - \nu_2(K) =  \nu(X) - \nu(X \sm K) =  \nu(K).
\end{align*}

By  Theorem \ref{STMSolid} $\nu = \mu$.
\end{proof}

\begin{remark}
Theorem \ref{STMdecom} is a locally compact version of the decomposition of a signed topological measure with finite norm 
into a difference of two topological measures  when the underlying space is 
compact, Hausdorff, connected, locally connected, and has genus 0.
See Section 7 in \cite{OrjanAlf:CostrPropQlf}. 
\end{remark}

\begin{remark}
(a) From Theorem \ref{STMdecom} we see that the decomposition of a signed topological measure 
into a difference of two topological measures is not unique.
This nonÐuniqueness of decomposition of a signed topological measure into a difference of topological measures 
is also demonstrated in Example 25 in \cite{OrjanAlf:CostrPropQlf}, where a signed topological measure on a compact space
is written as a difference of topological measures in two different ways.   \\
(b) If $X$ satisfies the conditions of Theorem \ref{STMdecom}, then
the family of topological measures on $X$ is a generating cone for 
the family of signed topological measures  with finite norms. 
\end{remark}

{\bf{Acknowledgments}}:
The author would like to thank the Department of Mathematics at the University of California Santa Barbara for its supportive environment.

%{\bf{Conflict of Interest}}:
%The authors declare that they have no conflict of interest.

\end{document}